\documentclass[graybox,footinfo]{svmult}
\smartqed
\usepackage{mathptmx}       
\usepackage{helvet}         
\usepackage{courier}        
\usepackage{type1cm}        
\usepackage{graphicx}       

\usepackage{array,colortbl}
\usepackage{amsmath,amsfonts,amssymb,bm} 

\newcommand{\cl}{{\mathcal L}}

\newcommand{\bx}{{\boldsymbol{x}}}

\newcommand{\balpha}{{\boldsymbol{\alpha}}}
\newcommand{\bgamma}{{\boldsymbol{\gamma}}}

\def\il{\left<}
\def\ir{\right>}

\def\a{\alpha }
\def\e{\varepsilon }
\def\phi{\varphi }
\def\epsilon{\varepsilon }
\def\g{\gamma }

\def\rho{\varrho }

\def\lstd{{\Lambda^{\rm std}}}
\def\lall{{\Lambda^{\rm all}}}
\newcount\refnum
\def\refx{\global\advance\refnum by 1 {\the\refnum . \ }}

\def\({\biggl( }
\def\){\biggr) }

\newcommand{\ch}{\mathcal{H}}

\newcommand{\bt}{\boldsymbol{t}}

\newcommand{\tK}{\widetilde{K}}
\newcommand{\hK}{\widehat{K}}

\newcommand{\naturals}{\mathbb{N}}
\newcommand{\reals}{\mathbb{R}}

\def\e{\varepsilon}

\newcommand{\stdall}{\vartheta}
\newcommand{\absnorm}{\psi}

\usepackage{microtype} 

\usepackage[colorlinks=true,linkcolor=black,citecolor=black,urlcolor=black]{hyperref}
\usepackage{bookmark}
\makeatletter 
  \providecommand*{\toclevel@author}{999}
  \providecommand*{\toclevel@title}{0}
\makeatother

\begin{document}
\titlerunning{Tractability of Approximation With Product Kernels}
\title*{Tractability of the Function Approximation Problem in Terms of the Kernel's Shape and Scale Parameters}
\author{Xuan Zhou \and Fred J.\ Hickernell}
\institute{
Xuan Zhou
\at Department of Applied Mathematics, Illinois Institute of Technology, Room E1-232,
10 W. 32$^{\text{nd}}$ Street, Chicago, IL 60616
\email{xzhou23@hawk.iit.edu}
\and
Fred J.\ Hickernell
\at Department of Applied Mathematics, Illinois Institute of Technology, Room E1-208,
10 W. 32$^{\text{nd}}$ Street, Chicago, IL 60616
\email{hickernell@iit.edu}
}

\maketitle

\abstract{
This article studies the problem of approximating functions belonging to a Hilbert space $\ch_d$ with a reproducing kernel of the form
$$\tK_d(\bx,\bt):=\prod_{\ell=1}^d \left(1-\a_\ell^2+\a_\ell^2K_{\g_\ell}(x_\ell,t_\ell)\right)\ \ \ \mbox{for all}
\ \ \ \bx,\bt\in\reals^d.$$
The $\a_\ell\in[0,1]$ are scale parameters, and the $\gamma_\ell>0$ are sometimes called shape parameters.  The reproducing kernel $K_{\g}$ corresponds to some Hilbert space of functions defined on $\reals$. The kernel $\tK_d$ generalizes the anisotropic Gaussian reproducing kernel, whose tractability properties have been established in the literature.  We present sufficient conditions on $\{\a_\ell \g_\ell\}_{\ell=1}^{\infty}$ under which polynomial tractability holds for function approximation problems on $\ch_d$. The exponent of strong polynomial tractability arises from bounds on the eigenvalues of a positive definite linear operator.
}

%

\section{Introduction}\label{intro}
This article addresses the problem of function approximation.
In a typical application we are given data of the
form $y_i=f(\bx_i)$ or $y_i=L_i(f)$ for $i=1,\ldots, n$.
That is, a function $f$ is sampled at the locations $\{\bx_1, \ldots, \bx_n\}$,
usually referred to as the \emph{data sites} or the \emph{design},
or more generally we know the values of $n$ linear functionals $L_1, \ldots, L_n$ applied to $f$. Here we assume that the domain of $f$ is a subset of
$\reals^d$.  The goal is to construct $A_n(f)$, a good approximation
to $f$ that is inexpensive to evaluate.

Algorithms for function approximation based on symmetric
positive definite kernels have arisen in both the numerical
computation literature \cite{Buh03a,Fas07a,SchWen06a,Wen05a}, and the
statistical learning literature
\cite{BerT-A04,CucZho07a,HasTibFrie01,RasWil06a,SchSmo02a,Ste99,SteChr08a,Wah90}.
These algorithms go by a variety of names, including radial
basis function methods \cite{Buh03a},
scattered data approximation \cite{Wen05a},
meshfree methods \cite{Fas07a}, (smoothing) splines \cite{Wah90},
kriging \cite{Ste99}, Gaussian process models \cite{RasWil06a} and
support vector machines \cite{SteChr08a}.

Many kernels commonly used in practice are associated with a sequence of shape parameters $\bgamma=\{\g_\ell\}_{\ell=1}^{\infty}$, which allows more flexibility in the function approximation problem. Examples of such kernels include the Mat\'{e}rn, the multiquadrics, the inverse multiquadrics, and the extensively studied Gaussian kernel (also known as the squared exponential kernel). The anisotropic stationary Gaussian kernel, is given by
\begin{equation} \label{anisoGauss}
\tK_d(\bx,\bt) := \E^{-\gamma_1^2 (x_1-t_1)^2 -\, \cdots \,
- \gamma_d^2 (x_d-t_d)^2} =
\prod_{\ell=1}^d \E^{-\gamma_\ell^2 (x_\ell-t_\ell)^2}\ \ \ \mbox{for all}
\ \ \ \bx,\bt\in\reals^d,
\end{equation}
where $\gamma_\ell$ is a positive shape parameter for each variable $x_\ell$. Choosing a small $\gamma_\ell$ has a beneficial effect on the rate of decay of the eigenvalues of the Gaussian kernel. The optimal choice of $\gamma_\ell$ is application dependent and much work has been spent on the quest for the optimal shape parameter. Note that taking $\gamma_\ell=\gamma$ for all $\ell$ will recover the isotropic Gaussian kernel.

For the Gaussian kernel \eqref{anisoGauss}, convergence rates with polynomial and tractability results are established in \cite{FasHicWoz12a}. These rates are summarized in Table~\ref{summarytable}. Note that the error of an algorithm $A_n$ in this context is the worst case error based on the following $\cl_2$ criterion:
\begin{equation} \label{l2normdef}
e^{\rm wor}(A_n):=\sup_{\|f\|_{\ch_d}\le1}\|f-A_n(f)\|_{\cl_2}, \qquad \|f\|_{\cl_2}:=\left(\int_{\reals^d}f^2(\bt)\,\rho_d(\bt)\,{\rm d}\bt\right)^{1/2},
\end{equation}
where $\rho_d$ is a probability density function with independent marginals, namely $\rho_d(\bx)=\rho_1(x_1) \cdots \rho_1(x_d)$.
For real $q$, the notation $\preceq n^{q}$ (with $n \to \infty$ implied) means
that for all $\delta > 0$
the quantity is bounded \emph{above} by $C_{\delta}n^{q+\delta}$ for all $n > 0$, where $C_{\delta}$
is some positive constant
that is independent of the sample size, $n$,
and the dimension, $d$, but
may depend on $\delta$.
The notation $\succeq n^{q}$
is defined analogously, and means that the quantity is
bounded from \emph{below} by $C_{\delta}n^{q-\delta}$
for all $\delta > 0$.
The notation $\asymp n^{q}$ means that the quantity is
both $\preceq n^{q}$  and $\succeq n^{q}$.
\begin{table}
\caption{Error decay rates for the Gaussian kernel as a function of sample size $n$}\label{summarytable}
\begin{tabular}{p{2.9cm}p{4.2cm}p{4.2cm}}
\hline\noalign{\smallskip}
Data Available & Absolute Error Criterion & Relative Error Criterion  \\
\noalign{\smallskip}\svhline\noalign{\smallskip}
Linear Functionals & $\asymp n^{-\max(r(\bgamma),1/2)}$ & $\asymp n^{-r(\bgamma)}$, if $r(\bgamma)>0$\\
Function values & $\preceq n^{-\max(r(\bgamma)/[1+1/(2r(\bgamma))],1/4)}$ & $\preceq n^{-r(\bgamma)/[1+1/(2r(\bgamma))]}$, if $r(\bgamma)>1/2$\\
\noalign{\smallskip}\hline\noalign{\smallskip}
\end{tabular}
\end{table}
The term $r(\bgamma)$ appearing in Table \ref{summarytable}
denotes the \emph{rate of convergence} to zero of the shape parameter
sequence $\bgamma$ and is defined by
\begin{equation} \label{wgammaform}
r(\bgamma):=\sup\bigg\{\beta>0\, \bigg |\ \sum_{\ell=1}^{\infty}
\gamma_\ell^{1/\beta} < \infty \bigg\}.
\end{equation}

The kernel studied in this article has a more general product form
\begin{equation}\label{GeneralKernel}
\tK_d(\bx,\bt)=\tK_{\balpha,\bgamma,d}(\bx,\bt) :=\prod_{\ell=1}^d \widehat K_{\a_\ell,\g_\ell}(x_\ell,t_\ell)\ \ \ \mbox{for all}\ \ \ \bx,\bt\in\reals^d,
\end{equation}
where $0\le\alpha_\ell \le1,\ \g_\ell>0$ and
\begin{equation}\label{hatK}
\hK_{\a,\g}(x,t):=1-\a^2+\a^2K_{\g}(x,t),\qquad x,t\in\reals.
\end{equation}
We assume that we know the eigenpair expansion of the one-dimensional kernel $K_{\g}$ in terms of its shape parameter $\g$. Many kernels in the numerical integration and approximation literature take the form of \eqref{GeneralKernel}, where $\a_\ell$ governs the vertical scale of the kernel across the $\ell$th dimension. In particular, taking $\a_\ell=1$ for all $\ell$ and $K_{\g}(x,t)=\exp(-\g^2(x-t)^2)$ recovers the anisotropic Gaussian kernel \eqref{anisoGauss}.

The goal of this paper is to extend the results in Table \ref{summarytable} to the kernel in \eqref{GeneralKernel}.  In essence we are are able to replace $r(\bgamma)$ by $\tilde r(\balpha,\bgamma)$ defined below in \eqref{tilder}.

Knowing the eigenpair expansion of $K_{\g}(x,t)$ does not give us explicit formulae for the eigenvalues and eigenfunctions of the kernel $\hK_{\a,\g}$. However, since the kernel \eqref{GeneralKernel} is of tensor product form and each factor is a convex combination of the constant kernel and a kernel with a known eigenpair expansion, we can derive upper and lower bounds of the eigenvalues of $\hK_{\a,\g}$ by approximating the corresponding linear operators by finite rank operators and applying some inequalities for eigenvalues of matrices. These bounds then yield tractability results for the general kernels.

\section{Function Approximation}
\subsection{Reproducing Kernel Hilbert Spaces}
Let $\ch_d=\ch(\tK_d)$ denote a reproducing kernel Hilbert space of real functions defined on $\reals^d$.  The goal is to approximate any function in $\ch_d$ given a finite number of
data about it.  The reproducing kernel $\tK_d:\reals^d\times \reals^d\to\reals$
is symmetric and positive definite.  It takes the form
\eqref{GeneralKernel}, where $K_{\g}$ satisfies the unit trace condition
\begin{equation}
\int_{\reals}K_{\g}(t,t)\,\rho_1(t)\,\D t=1 \qquad \forall \gamma>0.\label{unittrace}
\end{equation}

This condition implies that $\ch_d$ is continuously embedded in the space
$\cl_2=\cl_2(\reals^d,\rho_d)$ of square Lebesgue integrable functions,
where the $\cl_2$ norm was defined in \eqref{l2normdef}. Continuous embedding means that
$\|I_df\|_{\cl_2} = \|f\|_{\cl_2} \le \|I_d\|\ \|f\|_{\ch_d}$
for all $f\in \ch_d$.

Functions in $\ch_d$ are approximated by linear algorithms of the form
\begin{equation*}\label{linearnonadaptive}
(A_nf)\,(\bx):=\sum_{j=1}^nL_j(f)a_j(\bx)\ \ \
\mbox{for all}\ \ \  f\in \ch_d,\  \ \bx\in\reals^d\ \ \ \mbox{for all}
\ \ \ \bx\in\reals^d.
\end{equation*}
for some continuous linear functionals $L_j\in \ch_d^*$,
and functions $a_j\in \cl_2$. Note that for known functions $a_j$, the cost of computing $A_n(f)\,(x)$ is $\mathcal O(n)$, if we do not consider the cost of generating the data samples $L_j(f)$.
The linear functionals, $L_j$, used by an algorithm $A_n$ may either
come from the class of arbitrary bounded
linear functionals, $\lall = \ch_d^*$, or from
the class of function evaluations, $\lstd$.
The \emph{$n$th minimal worst case error} over all possible
algorithms is defined as
\begin{equation*}
e^{\text{wor-$\stdall$}}(n,\ch_d)
:=\inf_{A_n \ {\rm with}\ L_j\in\Lambda^{\stdall} }e^{\text{wor}}(A_n)
\quad \stdall \in \{{\rm std},{\rm all}\}.
\end{equation*}

\subsection{Tractability}While typical numerical analysis focuses on the rate of convergence, it does not take into consideration the effects of $d$.
The study of tractability arises in information-based complexity
and it considers how the error depends on the dimension, $d$,
as well as the number of data, $n$.

In particular, we would like to know how
$e^{\text{wor-$\stdall$}}(n,\ch_d)$
depends not only on $n$ but also on $d$.
Because of the focus on $d$-dependence,
the \emph{absolute} and \emph{normalized} error criteria
described in the previous section may lead to different answers.
For a given
positive $\e\in(0,1)$ we want to find
an algorithm $A_n$ with the smallest~$n$ for which the error does not
exceed $\e$ for the absolute error criterion, and does not exceed
$\e\,e^{\text{wor-$\stdall$}}(0,\ch_d)=\e\,\|I_d\|$
for the normalized error criterion. That is,
$$
n^{\text{wor-$\absnorm$-$\stdall$}}(\e,\ch_d)
= \min\left\{n\,|\
e^{\text{wor-$\stdall$}}(n,\ch_d)
\le\begin{cases}\e, & \absnorm = {\rm abs},\\
\e\,\|I_d\|, & \absnorm = {\rm norm},
\end{cases} \quad \right\}.
$$

Let $\mathcal{I}=\{I_d\}_{d \in \naturals}$ denote the sequence of
function approximation
problems. We say that $\mathcal{I}$ is~\emph{polynomially tractable} if and only if
there exist numbers $C$, $p$ and $q$ such that
\begin{equation} \label{tractdef}
n^{\text{wor-$\absnorm$-$\stdall$}}(\e,\ch_d)\le C\,d^{\,q}\,\e^{-p}\ \ \
\mbox{for all}\ \ \ d\in \naturals\ \ \mbox{and}\ \ \ \e\in(0,1).
\end{equation}
If $q=0$ above then we say that $\mathcal{I}$ is~\emph{strongly polynomially
tractable} and the infimum of~$p$ satisfying the bound above is
the~\emph{exponent} of strong polynomial tractability.

The essence of polynomial tractability is to guarantee that
a polynomial number of linear functionals is enough to satisfy
the function approximation problem to within $\e$.
Obviously, polynomial tractability depends on which class, $\lall$ or
$\lstd$, is considered and whether the absolute or normalized error is used.

The property of strong polynomial tractability is especially
challenging since then the number of linear functionals needed for an
$\e$-approximation is independent of $d$. Nevertheless, we provide here positive results on
strong polynomial tractability.

\section{Eigenvalues for the General Kernel}
Let us define the linear operator corresponding to any kernel $\tK_d$ as
\begin{equation*}\label{kernelOper}
Wf=\int_{\reals^d}f(\bt)\tK_d(\cdot,\bt)\rho_d(\bt) \, \D \bt\ \ \ \mbox{for all}\ \ \ f\in\ch_d.
\end{equation*}
It is known that $W$ is self-adjoint and positive definite if $\tK_d$ is a positive definite kernel. Moreover \eqref{unittrace} implies that $W$ is compact. Let us define the  eigenpairs of $W$ by
$(\lambda_{d,j},\eta_{d,j})$, where
the eigenvalues are ordered, $\lambda_{d,1}\ge\lambda_{d,2}\ge\cdots$, and
$$
W\,\eta_{d,j}=\lambda_{d,j}\,\eta_{d,j} \ \ \ \mbox{with}\ \ \
\il \eta_{d,j},\eta_{d,i}\ir_{\ch_d}=\delta_{i,j} \ \ \mbox{for all}\ \
i,j\in\naturals.
$$
Note also that for any $f\in \ch_d$ we have
$$
\il f,\eta_{d,j}\ir_{\cl_2}=\lambda_{d,j}\il f,\eta_{d,j}\ir_{\ch_d}.
$$
Taking $f=\eta_{d,i}$ we see that $\{\eta_{d,j}\}$
is a set of orthogonal functions in $\cl_2$. Letting
$$
\varphi_{d,j}=\lambda_{d,j}^{-1/2}\eta_{d,j}\ \ \
\mbox{for all}\ \ \ j\in\naturals,
$$
we obtain an orthonormal sequence $\{\varphi_{d,j}\}$ in $\cl_2$.
Since $\{\eta_{d,j}\}$ is a complete
orthonormal basis of $\ch_d$, we have
\begin{equation*}\label{formofkd}
\tK_d(\bx,\bt)=\sum_{j=1}^\infty\eta_{d,j}(\bx)\,\eta_{d,j}(\bt)=
\sum_{j=1}^\infty\lambda_{d,j}\,\varphi_{d,j}(\bx)\,
\varphi_{d,j}(\bt) \ \ \
\mbox{for all}\ \ \ \bx,\bt\in \reals^d.
\end{equation*}

To standardize the notation, we shall always write the eigenvalues of the linear operator corresponding to the  kernel $\tK_{d,\balpha,\bgamma}$ in \eqref{GeneralKernel} in a weakly decreasing order $\nu_{d,\balpha,\bgamma,1}\ge\nu_{d,\balpha,\bgamma,2}\ge\cdots$.
We drop the dependency on the dimension $d$ to denote the eigenvalues of the linear operator corresponding to the one-dimensional kernel $\hK_{\a,\g}$ in \eqref{hatK} by $\tilde\nu_{\a,\g,1}\ge\tilde\nu_{\a,\g,2}\ge\cdots$. Similarly the eigenvalues of the linear operator corresponding to the one-dimensional kernel $K_\g(x,t)$ are denoted by $\tilde\lambda_{\g,1}\ge\tilde\lambda_{\g,2}\ge\cdots$. A useful relation between the multivariate eigenvalues $\nu_{d,\balpha,\bgamma,j}$ and the univariate eigenvalues $\tilde\nu_{\a,\g,j}$ is given by \cite[Lemma 3.1]{FasHicWoz12a}:
\begin{equation*}\label{evRelation}
\sum_{j=1}^\infty\nu_{d,\balpha,\bgamma,j}^\tau=\prod_{\ell=1}^d\left(\sum_{j=1}^\infty\tilde\nu_{\a_\ell,\g_\ell,j}^\tau\right) \qquad \forall \tau>0.
\end{equation*}

We are interested in the high dimensional case where $d$ is large, and we want to establish convergence and tractability results when $\a_\ell$ and/or $\gamma_\ell$ tend to zero as $\ell\to\infty$. According to \cite{NovWoz08a}, strong polynomial tractability holds if the eigenvalues are bounded. The following lemma provides us some useful inequalities on eigenvalues of the linear operators corresponding to reproducing kernels.
\begin{lemma}\label{lm1}
Let $\ch(K_A),\ch(K_B),\ch(K_C)\subset\cl_2(\reals,\rho_1)$ be Hilbert spaces with symmetric positive definite reproducing kernels $K_A$, $K_B$ and $K_C$ such that
\begin{equation}\label{lmcon1}
\int_\reals K_\kappa(t,t)\rho_1(t) \, \D t<\infty,\ \ \ {\kappa}\in \{A,B,C\},
\end{equation}
and $K_C=aK_A+bK_B$, $a,b\ge0$.
Define the linear operators $W_A$, $W_B$, and $W_C$ by
\begin{equation*}
W_{\kappa}f=\int_{\reals}f(t)K_{\kappa}(\cdot,t)\rho_1(t) \, \D t,\ \ \ \mbox{for all}\ \ \ f\in\ch(K_{\kappa}),\ \ \ {\kappa}\in \{A,B,C\}.
\end{equation*}
Let the eigenvalues of the operators be sorted in a weakly decreasing order, i.e. $\lambda_{\kappa,1}\ge\lambda_{\kappa,2}\ge\cdots$.
Then these eigenvalues satisfy
\begin{equation}\label{lm1b}
\lambda_{C,i+j+1}\le a\lambda_{A,i+1}+b\lambda_{B,j+1},\ \ \ i,j=1,2,\ldots
\end{equation}
\begin{equation}\label{lm1c}
\lambda_{C,i}\ge\max(a\lambda_{A,i},b\lambda_{B,i}),\ \ \ i=1,2,\ldots
\end{equation}
\end{lemma}
\begin{proof}
Let $\{u_j\}_{j\in\naturals}$ be any orthonormal basis in $\cl_2(\reals,\rho_1)$. We assign the orthogonal projections $P_n$ given by
\begin{equation*}\label{Pn}
P_nx=\sum_{j=1}^n\il x,u_j\ir u_j,\ \ \ x\in\cl_2(\reals,\rho_1).
\end{equation*}
Since $W_A$ is compact due to \eqref{lmcon1}, it can be shown that $\|(I-P_n)W_A\|\to0$ as $n\to\infty$, where the operator norm
$$\|(I-P_n)W_A\|:=\sup_{\|x\|\le1}\|(I-P_n)W_Ax\|_{\cl_2(\reals,\rho_1)}.$$
Furthermore \cite[Lemma 11.1 ($OS_2$)]{Pie80} states that for every pair $T_1,T_2:X\to Y$ of compact operators we have
$|s_j(T_1)-s_j(T_2)|\le\|T_1-T_2\|,\ j\in\naturals,$
where the singular values $s_j(T_k),k=1,2$ are the square roots of the eigenvalues $\lambda_j(T_k^*T_k)$ arranged in a weakly decreasing order, thus $s_j(T_k)=\sqrt{\lambda_j(T_k^*T_k)}$. Now we can bound
\begin{align*}
|s_j(W_A)-s_j(P_nW_AP_n)|&\le|s_j(W_A)-s_j(P_nW_A)|+|s_j(P_nW_A)-s_j(P_nW_AP_n)|\\
&\le\|W_A-P_nW_A\|+\|P_nW_A-P_nW_AP_n\|\\
&\le\|(I-P_n)W_A\|+\|W_A(I-P_n)\|\to0
\end{align*}
as $n\to\infty$. Thus the eigenvalues $\lambda_{P_nW_AP_n,j}\to\lambda_{W_A,j}$ for all $j$ as $n\to\infty$. Similarly this applies to the operators $W_B$ and $W_C$. Note that we have
$$P_nW_CP_n=aP_nW_AP_n+bP_nW_BP_n$$
and these finite rank operators correspond to self-adjoint matrices. These matrices are symmetric and positive definite because the kernels are symmetric and positive definite. Since the inequalities \eqref{lm1b} and \eqref{lm1c} hold for the eigenvalues of symmetric positive definite matrices, they also hold for the operators corresponding to symmetric and positive definite kernels.
\qed
\end{proof}

We are now ready to present the main results of this article in the following two sections.

\section{Tractability for the Absolute Error Criterion}\label{secAbs}
We now consider the function approximation problem for Hilbert spaces $\ch_d=\ch(\widetilde K_d)$ with a general kernel using the absolute error criterion. From the discussion of eigenvalues in the previous section and from \eqref{unittrace} it follows that
\begin{equation}\label{evsum}
\sum_{j=1}^\infty\tilde\lambda_{\g,j}=\int_\reals K_{\g}(t,t)\rho_1(t)\, \D t=1, \qquad \forall \g>0.
\end{equation}
We want to verify whether polynomial tractability holds, namely whether \eqref{tractdef} holds.

\subsection{Arbitrary Linear Functionals}
We first analyze the class $\lall$ and  polynomial
tractability. Similar to \eqref{wgammaform}, let us define the rate of decay of scale and shape parameters $\tilde r(\balpha,\bgamma)$ as
\begin{equation}\label{tilder}
\tilde r(\balpha,\bgamma)=\sup\bigg\{\beta>0\bigg|\sum_{\ell=1}^\infty(\a_\ell\gamma_\ell)^{1/\beta}<\infty\bigg\}
\end{equation}
with the convention that the supremum of the empty set is taken to be zero.
\begin{theorem}\label{ThmAbsAll}
Consider the function approximation problem $\mathcal{I}=\{I_d\}_{d \in \naturals}$
for Hilbert spaces for the class $\lall$ and
the absolute error criterion with the kernels \eqref{GeneralKernel} satisfying \eqref{evsum}. Let $\tilde r(\balpha,\bgamma)$ be given by \eqref{tilder}.
If there exist constants $C_1,C_2,C_3>0$, which are independent of $\gamma$ but may depend on $\tilde r(\balpha,\bgamma)$ and $\sup\{\g_\ell|\ell\in\naturals\}$, such that
\begin{eqnarray}
\int_{\reals^2}K_{\g}(x,t)\rho_1(x)\rho_1(t)dxdt&\ge&1-C_1\g^2,\label{gtc1}\\
C_2\le\sum_{j=2}^\infty\left(\frac{\tilde\lambda_{\g,j}}{\gamma^2}\right)^{\frac1{2\tilde r(\balpha,\bgamma)}}&\le& C_3\label{gtc2}
\end{eqnarray}
hold for all $0 < \gamma < \sup\{\g_\ell|\ell\in\naturals\}$, then it follows that
\begin{itemize}
\item $\mathcal{I}$
is strongly polynomially tractable
with exponent
$$
p^{\rm all}=\min\left(2,\frac1{\tilde r(\balpha,\bgamma)}\right).
$$
\item
For all $d\in\naturals$ we have
\begin{eqnarray*}
e^{\text{\rm wor-all}}(n,\ch_d)&\preceq &n^{-1/p^{\rm all}} = n^{-\max(\tilde r(\balpha,\bgamma),1/2)} \quad n \to \infty,\\
n^{\text{\rm wor-abs-all}}(\e,\ch_d)& \preceq &
\e^{-p^{\rm all}} \quad \e \to 0,
\end{eqnarray*}
where $\preceq n^{q}$ with $n \to \infty$ was defined in Section \ref{intro}, and $\preceq\e^{q}$ with $\e \to 0$ is analogous to $\preceq (1/\e)^{-q}$ with $1/\e \to \infty$.
\item
For the isotropic kernel with $\a_\ell=\a$ and $\g_\ell=\g$ for all $\ell$, the exponent of
strong tractability is $2$. 
Furthermore strong polynomial tractability is equivalent to polynomial tractability.
\end{itemize}
\end{theorem}
\begin{proof}
From \cite[Theorem 5.1]{NovWoz08a} it follows that $\mathcal I$ is strongly polynomially tractable if and only if there exist two positive numbers $c_1$ and $\tau$ such that
\begin{equation}\label{ptCond}
c_2:=\sup_{d\in\naturals}
\left(\sum_{j=\lceil
    c_1\rceil}^\infty\nu^\tau_{d,\balpha,\bgamma,j}\right)^{1/\tau}<\infty,
\end{equation}
Furthermore, the exponent $p^{\rm all}$ of strong polynomial
tractability is the infimum of $2\tau$ for which this
condition holds. Obviously \eqref{ptCond} holds for $c_1=1$ and $\tau=1$ because
\begin{align*}
\sum_{j=1}^\infty\nu_{d,\balpha,\bgamma,j}&=\prod_{\ell=1}^d\left(\sum_{j=1}^\infty\tilde\nu_{\a_\ell,\g_\ell,j}\right)=\prod_{\ell=1}^d\left(\int_\reals[1-\a_\ell^2+\a_\ell^2K_{\g_\ell}(t,t)]\rho_1(t)dt\right)\\
&=\prod_{\ell=1}^d\left(1-\a_\ell^2+\a_\ell^2\right)=1.
\end{align*}
This shows that $p^{\rm all}\le2$.

Take now $\tilde r(\balpha,\bgamma)>0$. Consider first the case $d=1$ for simplicity. Then kernel $\tK_{d,\balpha,\bgamma}$ in \eqref{GeneralKernel} becomes
$\hK_{\a,\g}.$
We will show that for $\tau=1/(2\tilde r(\balpha,\bgamma))$, the eigenvalues of $\widehat K_{\a,\g}$ satisfy
\begin{equation}\label{evu}
\sum_{j=1}^\infty\tilde\nu_{\a,\g,j}^{\tau}\le 1+C_\text U(\a\g)^{2\tau},
\end{equation}
where the constant $C_\text U$ does not depend $\a$ or $\g$. Since all the eigenvalues of $K_\g$ are non-negative, we clearly have for the first eigenvalue of $K_\g$,
\begin{equation} \label{ev1u}
\tilde\nu_{\a,\g,1}\le1.
\end{equation}
On the other hand, \eqref{gtc1} gives the lower bound of the first eigenvalue of $\widehat K_{\a,\gamma}$
\begin{align}
\tilde\nu_{\a,\g,1}&\ge\int_{\reals^2}\widehat K_{\a,\gamma}(x,t)\rho_1(x)\rho_1(t) \, \D t \D x=\int_{\reals^2}\left(1-\a^2+\a^2 K_\gamma(x,t)\right)\rho_1(x)\rho_1(t) \, \D t \D x \nonumber\\
&=1-\a^2+\a^2\int_{\reals^2}K_\gamma(x,t)\rho_1(x)\rho_1(t) \, \D t \D x \ge1-C_1(\a\gamma)^2.\label{ev1l}
\end{align}
It follows from \eqref{evsum} that
\begin{equation}\label{ev2u}
\tilde\nu_{\a,\g,2}\le C_1(\a\g)^2.
\end{equation}
For $j\ge3$, the upper bound of $\tilde\nu_{\a,\g,j}$ is given by \eqref{lm1b} with $i=1$:
\begin{equation}\label{nll}
\tilde\nu_{\a,\g,j}\le\a^2\tilde\lambda_{\g,j-1},
\end{equation}
which in turn yields
\begin{equation}\label{ev3u}
\sum_{j=3}^\infty\tilde\nu_{\a,\g,j}^{\tau}\le \a^{2\tau} \sum_{j=3}^\infty\tilde\lambda_{\g,j-1}^{\tau}\le C_3(\a\g)^{2\tau}
\end{equation}
by \eqref{gtc2}. Combining \eqref{ev1u}, \eqref{ev2u} and \eqref{ev3u} gives \eqref{evu},
where the constant $C_\text U=C_1^\tau+C_3$.

The lower bound we want to establish is that for $\tau<1/(2\tilde r(\balpha,\bgamma))$,
\begin{equation}\label{evl}
\sum_{j=1}^\infty\tilde\nu_{\a,\g,j}^\tau\ge 1+C_\text L(\a\g)^{2\tau} \quad \text{if} \quad \a\g < \left(\frac{C_2}{2C_1}\right)^{1/[2(1-\tau)]},
\end{equation}
where $C_\text L:=C_2/2$.
It follows from \eqref{ev1l} that
\begin{equation}\label{ev1tl}
\tilde\nu_{\a,\g,1}^\tau\ge\tilde\nu_{\a,\g,1}\ge1-C_1(\a\gamma)^2.
\end{equation}
In addition we apply the eigenvalue inequality \eqref{lm1b} to obtain
\begin{equation}
\tilde\nu_{\a,\g,j}\ge\a^2\tilde\lambda_{\g,j},\ \ \ j=2,3,\ldots\nonumber
\end{equation}
which in turn gives
\begin{equation}\label{ev2tl}
\sum_{j=2}^\infty\tilde\nu_{\a,\g,j}^\tau\ge\a^{2\tau}\sum_{j=2}^\infty\tilde\lambda_{\g,j}^\tau\ge C_2(\a\g)^{2\tau},
\end{equation}
where the last inequality follows from \eqref{gtc2}. Inequalities \eqref{ev1tl} and \eqref{ev2tl} together give
$$\sum_{j=1}^\infty\tilde\nu_{\a,\g,j}^\tau\ge1-C_1(\a\gamma)^2+C_2(\a\g)^{2\tau} \ge 1+(C_2/2)(\a\g)^{2\tau} $$
under the condition in \eqref{evl} on small enough $\a\g$.  Thus we obtain \eqref{evl}.

For the multivariate case, the sum of the $\tau$-th power of the eigenvalues is bounded from above for $\tau=1/(2\tilde r(\balpha,\bgamma))$ because
\begin{align}
\sum_{j=1}^\infty\nu_{d,\balpha,\bgamma,j}^{\tau}&=\prod_{\ell=1}^d\left(\sum_{j=1}^\infty\tilde\nu_{\a_\ell,\g_\ell,j}^{\tau}\right)\le\prod_{\ell=1}^\infty\left(1+C_\text U(\a_\ell\gamma_\ell)^{2\tau}\right)\nonumber\\
&=\exp\left(\sum_{\ell=1}^\infty\ln\left(1+C_\text U(\a_\ell\gamma_\ell)^{2\tau}\right)\right)\le\exp\left(C_\text U\sum_{\ell=1}^\infty (\a_\ell\gamma_\ell)^{2\tau}\right)<\infty.\label{evtu}
\end{align}
This shows that
$p^{\rm all}\le1/\tilde r(\balpha,\bgamma).$

We now consider the lower bound in the multivariate case and define the set $A$ by
$$A=\left\{\ell\left| \a_\ell \g_\ell  < \left(\frac{C_2}{2C_1}\right)^{1/[2(1-\tau)]} \right.\right\}.$$
Then
\begin{align*}
\sup_{d\in\naturals}\left(\sum_{j=1}^\infty\nu^\tau_{d,\balpha,\bgamma,j}\right)&=\prod_{\ell=1}^\infty\left(\sum_{j=1}^\infty\tilde\nu_{\a_\ell,\g_\ell,j}^{\tau}\right)=\prod_{\ell\in A}\left(\sum_{j=1}^\infty\tilde\nu_{\a_\ell,\g_\ell,j}^{\tau}\right)\prod_{\ell\in\naturals\setminus A}\left(\sum_{j=1}^\infty\tilde\nu_{\a_\ell,\g_\ell,j}^{\tau}\right).
\end{align*}
We want to show that this supremum is infinite for $\tau<1/(2\tilde r(\balpha,\bgamma))$.  We do this by proving that the first product on the right is infinite. Indeed for $\tau<1/(2\tilde r(\balpha,\bgamma))$,
\begin{align*}
\prod_{\ell\in A}\left(\sum_{j=1}^\infty\tilde\nu_{\a_\ell,\g_\ell,j}^{\tau}\right)&\ge\prod_{\ell\in A}\left[1+C_\text L(\a_\ell\g_\ell)^{2\tau}\right]\ge1+C_\text L\sum_{\ell\in A}(\a_\ell\g_\ell)^{2\tau} = \infty.
\end{align*}
Therefore,
$p^{\rm all}\ge1/\tilde r(\balpha,\bgamma)$, which establishes the formula for $p^{\rm all}$. The estimates on $e^{\text{wor-all}}(n,\ch_d)$ and $n^{\text{wor-abs-all}}(\e,\ch_d)$ follow from the definition of strong tractability.

Finally, the exponent of strong tractability is 2 for the isotropic kernel because $\tilde r(\balpha,\bgamma)=0$ in this case. To prove that strong polynomial tractability is equivalent
to polynomial tractability, it is enough to show that
polynomial tractability implies strong polynomial tractability.
From \cite[Theorem 5.1]{NovWoz08a} we know that
polynomial tractability holds if and only if
there exist numbers $c_1>0$, $q_1\ge0$, $q_2\ge0$ and $\tau>0$
such that
$$
c_2:=\sup_{d\in\naturals}\left\{d^{-q_2}
\left(\sum_{j=\lceil
    C_1\,d^{\,q_1}\rceil}^\infty\lambda^\tau_{d,j}\right)^{1/\tau}
\right\}<\infty.
$$
If so, then
$$
n^{\text{wor-abs-all}}(\e,\ch_d)\le(c_1+c_2^\tau)\,d^{\,\max(q_1,q_2\tau)}\,
\e^{-2\tau}
$$
for all $\e\in(0,1)$ and $d\in\naturals$. Note that for all $d$ we
have
\begin{align*}
d^{-q_2\tau}\left(\sum_{j=1}^\infty\tilde\nu_{\a,\g,j}^\tau\right)^d-d^{-q_2\tau}(\lceil c_1\rceil-1)\tilde\nu_{\a,\g,1}^{\tau d}\le c_2^\tau<\infty.
\end{align*}
This implies that $\tau\ge1$. On the other hand, for $\tau=1$ we can
take $q_1=q_2=0$ and arbitrarily small $C_1$, and
obtain strong tractability.
This completes the proof.
\qed
\end{proof}

Theorem~\ref{ThmAbsAll} states that the exponent of strong polynomial tractability is at most $2$, while for all shape parameters for which $\tilde r(\bgamma)>1/2$ the exponent is smaller than $2$. Again, although the rate of convergence of
$e^{\text{wor-all}}(n,\ch_d)$ is always excellent, the
dependence on $d$ is eliminated only at the expense of the exponent
which must be roughly  $1/p^{\rm all}$.  Of course, if we take an
exponentially decaying sequence of the products of scale parameters and shape parameters, say,
$\a_\ell\gamma_\ell=q^{\,\ell}$ for some
$q\in(0,1)$, then $\tilde r(\bgamma)=\infty$ and $p^{\rm all}=0$. In this
case, we have an excellent rate of convergence without any dependence on $d$.

\subsection{Only Function Values}

The tractability results for the class $\Lambda^{\text{std}}$ are stated in the following theorem.
\begin{theorem}\label{ThmAbsStd}
Consider the function approximation problem $\mathcal{I}=\{I_d\}_{d \in \naturals}$
for Hilbert spaces for the class $\lstd$ and
the absolute error criterion with the kernels \eqref{GeneralKernel} satisfying \eqref{evsum}. Let $\tilde r(\balpha,\bgamma)$ be given by \eqref{tilder}.
If there exist constants $C_1,C_2,C_3>0$, which are independent of $\gamma$ but may depend on $\tilde r(\balpha,\bgamma)$ and $\sup\{\g_\ell|\ell\in\naturals\}$, such that \eqref{gtc1} and \eqref{gtc2} are satisfied for all $0 < \gamma < \sup\{\g_\ell|\ell\in\naturals\}$, then
\begin{itemize}
\item $\mathcal{I}$
is strongly polynomially tractable
with exponent of strong polynomial tractability at most $4$.
For all $d\in\naturals$ and $\e\in(0,1)$ we have
\begin{eqnarray*}
e^{\text{\rm wor-std}}(n,\ch_d)&\le&\frac{\sqrt{2}}{n^{1/4}}
\,\left(1+\frac1{2\sqrt{n}}\right)^{1/2},\\
n^{\rm wor-abs-std}(\e,\ch_d)
&\le& \left\lceil\frac{(1+\sqrt{1+\e^2})^2}{\e^{4}}\right\rceil.
\end{eqnarray*}
\item
For the isotropic kernel with $\a_\ell=\a$ and $\g_\ell=\g$ for all $\ell$, the exponent of
strong tractability is at least $2$ and strong polynomial tractability is equivalent to polynomial tractability.
\end{itemize}
Furthermore if $\tilde r(\balpha,\bgamma)>1/2$, then
\begin{itemize}
\item $\mathcal{I}$
is strongly polynomially tractable
with exponent of strong polynomial tractability at most
$$p^\text{\rm std}=\frac1{\tilde r(\balpha,\bgamma)}+\frac1{2\tilde r^2(\balpha,\bgamma)}=p^\text{\rm all}+\frac12(p^\text{\rm all})^2<4.$$
\item For all $d\in\naturals$ we have
\begin{eqnarray*}
e^{\text{\rm wor-std}}(n,\ch_d)&\preceq &n^{-1/p^{\rm std}} = n^{-\tilde r(\balpha,\bgamma)/[1+1/(2\tilde r(\balpha,\bgamma))]} \quad n \to \infty,\\
n^{\text{\rm wor-abs-std}}(\e,\ch_d)& \preceq &
\e^{-p^{\rm std}} \quad \e \to 0.
\end{eqnarray*}
\end{itemize}
\end{theorem}
\begin{proof}
The same proofs as for \cite[Theorem 5.3 and 5.4]{FasHicWoz12a} can be used. We only need to show that the assumption of \cite[Theorem 5]{KuoWasWoz09a}, which is used in \cite[Theorem 5.4]{FasHicWoz12a}, is satisfied. It is enough to show that there exists $p>1$ and $B>0$ such that for any $n\in\naturals$,
\begin{equation}\label{a3p}
\nu_{d,\balpha,\bgamma,n}\le\frac{B}{n^p}.
\end{equation}
Take $\tau=1/(2\tilde r(\balpha,\bgamma))$. Since the eigenvalues $\tilde\lambda_{\g_\ell,n}$ are ordered, we have for $n\ge2$,
\begin{equation*}
\tilde\lambda_{\gamma_\ell,n}^{\tau}\le\frac{1}{n-1}\sum_{j=2}^n\tilde\lambda_{\gamma_\ell,j}^{\tau}\le\frac{1}{n-1}\sum_{j=2}^\infty\tilde\lambda_{\gamma_\ell,j}^{\tau}\le\frac{C_3\g_\ell^{2\tau}}{n-1},
\end{equation*}
where the last inequality follows from \eqref{gtc2}. Raising to the power $1/\tau_0$ gives
\begin{equation*}
\tilde\lambda_{\gamma_\ell,n}\le\g_\ell^2\left(\frac{C_3}{n-1}\right)^{1/\tau}.
\end{equation*}
Furthermore \eqref{nll} implies that for $n\ge3$,
\begin{align*}
\tilde\nu_{\a_\ell,\gamma_\ell,n}&\le\a_\ell^2\tilde\lambda_{\gamma_\ell,n-1}\le\a_\ell^2\g_\ell^2\left(\frac{C_3}{n-2}\right)^{1/\tau}=\a_\ell^2\g_\ell^2C_3^{1/\tau}\left(\frac n{n-2}\right)^{1/\tau}\left(\frac1{n^{1/\tau}}\right)\\
&\le\frac{\a_\ell^2\g_\ell^2(3C_3)^{1/\tau}}{n^{1/\tau}}.
\end{align*}
Since $\tilde\nu_{\a_\ell,\g_\ell,n}\le1$ for all $n\in\naturals$, we have that for all $1\le\ell\le d$ and $n\ge3$,
\begin{equation*}
\nu_{d,\balpha,\bgamma,n}\le\tilde\nu_{\a_\ell,\g_\ell,n}\le\frac{C_4}{n^p},
\end{equation*}
where $C_4=\a_\ell^2\g_\ell^2(3C_3)^{1/\tau}$ and $p=1/\tau>1$. For $n=1$ and $n=2$, we can always find $C_5$ large enough such that
$\nu_{d,\balpha,\bgamma,n}\le{C_5}/{n^p}.$
Therefore \eqref{a3p} holds for $B=\max\{C_4,C_5\}$.
\qed
\end{proof}

Note that \eqref{a3p} can be easily satisfied for many kernels used in practice. This theorem implies that for large $\tilde r(\balpha,\bgamma)$, the exponents of strong polynomial
tractability are nearly the same for both classes $\lall$ and $\lstd$.
For an exponentially decaying sequence of shape parameters, say,
$\a_\ell\gamma_\ell=q^{\,\ell}$
for some $q\in(0,1)$, we have $p^{\rm all}=p^{\rm std}=0$, and
the rates of convergence are excellent and independent of $d$.

\section{Tractability for the Normalized Error Criterion}\label{secNor}
We now consider the function approximation problem for Hilbert spaces
$\ch_d(\tK_d)$ with a general kernel for the normalized error
criterion. That is, we want to find the smallest $n$ for which
$$
e^{\text{wor-$\stdall$}}(n,\ch_d)\le \e\,\|I_d\|,
\qquad \stdall \in \{{\rm std},{\rm all}\}.
$$
Note that $\|I_d\|=\sqrt{\nu_{d,\balpha,\bgamma,1}}\le1$ and it can be
exponentially small in $d$. Therefore
the normalized error criterion may be much harder
than the absolute error criterion. It follows from \cite[Theorem 6.1]{FasHicWoz12a} that for the normalized error criterion, lack of polynomial tractability holds for the isotropic kernel for the class $\Lambda^{\text{all}}$ and hence for the class $\Lambda^{\text{std}}$.
\subsection{Arbitrary Linear Functionals}

We do not know polynomial tractability holds for kernels with $0<\tilde r(\balpha,\bgamma)<1/2$. For $\tilde r(\balpha,\bgamma)\ge1/2$, we have the following theorem.
\begin{theorem}\label{ThmNorAll}
Consider the function approximation problem $\mathcal{I}=\{I_d\}_{d \in \naturals}$
for Hilbert spaces for the class $\lstd$ and
the relative error criterion with the kernels \eqref{GeneralKernel} satisfying \eqref{evsum}. Let $\tilde r(\balpha,\bgamma)$ be given by \eqref{tilder} and $\tilde r(\balpha,\bgamma)\ge1/2$.
If there exist constants $C_1,C_2,C_3>0$, which are independent of $\gamma$ but may depend on $\tilde r(\balpha,\bgamma)$ and $\sup\{\g_\ell|\ell\in\naturals\}$, such that \eqref{gtc1} and \eqref{gtc2} are satisfied for all $0 < \gamma < \sup\{\g_\ell|\ell\in\naturals\}$,  then
\begin{itemize}
\item
$\mathcal{I}$
is strongly polynomially tractable with exponent of strong polynomial tractability
$$
p^{\rm all}=\frac1{\tilde r(\balpha,\bgamma)}.
$$
\item
For all $d\in\naturals$ we have
\begin{eqnarray*}
e^{\text{\rm wor-all}}(n,\ch_d)&\preceq &\|I_d\|n^{-1/p^{\rm all}} = n^{-\tilde r(\bgamma)} \quad n \to \infty,\\
n^{\text{\rm wor-abs-all}}(\e,\ch_d)& \preceq &
\e^{-p^{\rm all}} \quad \e \to 0.
\end{eqnarray*}
\end{itemize}
\end{theorem}
\begin{proof} From
\cite[Theorem 5.2]{NovWoz08a} we know that strong polynomial tractability
holds if and only if there exits a positive number $\tau$ such that
$$
c_2:=\sup_{d}
\sum_{j=1}^\infty
\left(\frac{\nu_{d,\balpha,\bgamma,j}}{\nu_{d,\balpha,\bgamma,1}}\right)^\tau=\sup_{d}\left\{\frac1{\nu_{d,\balpha,\bgamma,1}^\tau}\sum_{j=1}^\infty\nu_{d,\balpha,\bgamma,j}^\tau\right\}<\infty.
$$
If so, then $n^{\text{wor-nor-all}}(\e,\ch_d)\le c_2\,\e^{-2\tau}$
 for all $\e\in(0,1)$ and $d\in\naturals$,
and the exponent of strong polynomial tractability is the infimum of $2\tau$
for which $c_2<\infty$.

For all $d\in\naturals$, we have $\sum_{j=1}^\infty\nu_{d,\balpha,\bgamma,j}^\tau<\infty$ for $\tau=1/(2\tilde r(\balpha,\bgamma))$ from $\eqref{evtu}$. It remains to note that $\sup_d\{1/\nu_{d,\balpha,\bgamma,1}^\tau\}<\infty$ if and only if $\sup_d\{1/\nu_{d,\balpha,\bgamma,1}\}<\infty$. Furthermore note that \eqref{ev1l} implies that
$$\sup_d\left\{\frac1{\nu_{d,\balpha,\bgamma,1}}\right\}\le\prod_{\ell=1}^\infty\frac1{1-C_1(\a_\ell\g_\ell)^{2}}.$$
Clearly, $\tilde r(\balpha,\bgamma)\ge1/2$ implies that $\sum_{\ell=1}^\infty(\a_\ell\g_\ell)^{2}<\infty$, which yields $c_2<\infty$.
This also proves that
$p^{\rm all}\le1/\tilde r(\balpha,\bgamma)$.
The estimates on $e^{\text{wor-all}}(n,\ch_d)$
and $n^{\text{wor-nor-all}}(\e,\ch_d)$ follow from the definition of strong
tractability.
\qed
\end{proof}

\subsection{Only Function Values}

We now turn to the class $\lstd$. We do not know if polynomial tractability holds for the class $\lstd$ for $0<\tilde r(\balpha,\bgamma)\le1/2$. For $\tilde r(\balpha,\bgamma)>1/2$, we have the following theorem.
\begin{theorem}\label{ThmNorStd}
Consider the function approximation problem $\mathcal{I}=\{I_d\}_{d \in \naturals}$
for Hilbert spaces with the kernel \eqref{GeneralKernel} for the class $\lstd$ and
the normalized error criterion. Let $\tilde r(\balpha,\bgamma)$ be given by \eqref{tilder} and $\tilde r(\balpha,\bgamma)>1/2$.
If there exist constants $C_1,C_2,C_3>0$, which are independent of $\gamma$ but may depend on $\tilde r(\balpha,\bgamma)$ and $\sup\{\g_\ell|\ell\in\naturals\}$, such that \eqref{gtc1} and \eqref{gtc2} are satisfied for all $0 < \gamma < \sup\{\g_\ell|\ell\in\naturals\}$, then
\begin{itemize}
\item $\mathcal{I}$
is strongly polynomially tractable
with exponent of strong polynomial tractability at most
$$p^\text{\rm std}=\frac1{\tilde r(\balpha,\bgamma)}+\frac1{2\tilde r^2(\balpha,\bgamma)}=p^\text{\rm all}+\frac12(p^\text{\rm all})^2<4.$$
For all $d\in\naturals$ we have
\begin{eqnarray*}
e^{\text{\rm wor-std}}(n,\ch_d)&\preceq &n^{-1/p^{\rm std}} \quad n \to \infty,\\
n^{\text{\rm wor-nor-std}}(\e,\ch_d)& \preceq &
\e^{-p^{\rm std}} \quad \e \to 0.
\end{eqnarray*}
\end{itemize}
\end{theorem}
\begin{proof}
The initial error is
$$\|I_d\|\ge\prod_{\ell=1}^d(1-C_1(\a_\ell\g_\ell)^2)^{1/2}=\exp\left(\mathcal O(1)-\frac12\sum_{\ell=1}^d(\a_\ell\g_\ell)^2\right).$$
$\tilde r(\balpha,\bgamma)>1/2$ implies that $\|I_d\|$ is uniformly bounded from below by a positive number. This shows that there is no difference between the absolute and normalized error criteria. This
means that we can apply Theorem~\ref{ThmAbsStd} for the
class $\lstd$ with $\e$ replaced by $\e\|I_d\|=\Theta(\e)$.
This completes the proof.
\qed
\end{proof}

\begin{acknowledgement}
We are grateful for many fruitful discussions with Peter Math\'e and several other colleagues.  This work was partially supported by US National Science Foundation grants DMS-1115392 and DMS-1357690.
\end{acknowledgement}

\bibliographystyle{spmpsci}
\bibliography{FJH22,FJHown22,XZ01}

\end{document}